\documentclass[11pt]{amsart}% [12pt] indica la dimensione del teso, {report} il tipo di documento
\usepackage{etex}
\author{Enrico Fatighenti, Luca Rizzi and Francesco Zucconi}
\title{Weighted Fano Varieties and infinitesimal Torelli problem}
\usepackage[english]{babel}
\usepackage{latexsym}
\usepackage[utf8]{inputenc}
\usepackage{babel}
\usepackage{fancyhdr}
\usepackage{mathrsfs}
\usepackage{textcomp}
\usepackage{geometry}
\usepackage{longtable}
\usepackage[T1]{fontenc}
\usepackage{textcomp}
\usepackage[all,cmtip]{xy}
\usepackage{epsfig}
 \usepackage{amsmath,amsfonts,amssymb,amsthm}
\usepackage{graphics}
\usepackage{graphicx}
\usepackage{pictexwd, dcpic}
\usepackage{mathrsfs}
\usepackage{caption}
\usepackage{url}


\newcommand{\C}{\mathbb{C}}
\newcommand{\A}{\mathbb{A}}
\newcommand{\Q}{\mathbb{Q}}
\newcommand{\Z}{\mathbb{Z}}

\newcommand{\PP}{\mathbb{P}}
\newcommand{\mbP}{\mathbb{P}}
\newcommand{\of}{\mathcal{O}}

\DeclareMathOperator{\Hom}{Hom}

\DeclareMathOperator{\ddeg}{deg}
\DeclareMathOperator{\dd}{d}

\DeclareMathOperator{\Sym}{Sym}

\DeclareMathOperator{\ddim}{dim}
\DeclareMathOperator{\prim}{prim}
\newtheorem{thm}{Theorem}[section]
\newtheorem{corollary}[thm]{Corollary}
\newtheorem{question}[thm]{Question}

\newtheorem{lemma}[thm]{Lemma}
\newtheorem{proposition}[thm]{Propostion}
\newtheorem{defi}[thm]{Definition}
\newtheorem{rmk}[thm]{Remark}

\newtheorem{ex}[thm]{Example}

\begin{document}
\maketitle



%\title{Weighted Fano Varieties and infinitesimal Torelli problem}

%\author[warwick]{Enrico Fatighenti}
%\author[udine]{Luca Rizzi } 
%\author[udine]{Francesco Zucconi}
%\address[warwick]{Mathematics Institute Zeeman Building, University of Warwick, Coventry CV4 7AL, Warwick, England}
%\address[udine]{D.M.I.F., University of Udine, Via delle Scienze 206, Udine, Italy}

\begin{abstract}
We solve the infinitesimal Torelli problem for $3$-dimensional quasi-smooth ${\mathbb{Q}}$-Fano hypersurfaces with at worst terminal singularities. We also find infinite chains of double coverings of increasing dimension which alternatively distribute themselves in examples and counterexamples for the infinitesimal Torelli claim and which share the analogue, and in some cases the same, Hodge-diagram properties as the length $3$ Gushel-Mukai chain of prime smooth Fanos of coindex $3$ and degree 10.

\end{abstract}

\section{Introduction}

Minimal model program suggests us to formulate inside the category of varieties with terminal singularities many questions which were initially asked for smooth varieties. The construction of the period map, and the related Torelli type problems are definitely among these problems. 

\subsection{Infinitesimal Torelli} A ${\mathbb{Q}}$-Fano variety is a projective variety $X$ such that $X$ has at worst ${\mathbb{Q}}$-factorial terminal singularities, $-K_{X}$ is ample and ${\rm{Pic}}(X)$ has rank $1$; cf. \cite{cpr}. In this paper we give a full answer to the infinitesimal Torelli problem in the case of quasi-smooth ${\mathbb{Q}}$-Fano hypersurfaces of dimension $3$ with terminal  singularities and with Picard number $1$.

In subsections \ref{subsectioninfinitesimaltorelli} of this paper the reader can find a basic dictionary and up to date references needed to understand the statement of the infinitesimal Torelli problem and where to find the meaning of infinitesimal variation of Hodge structures (IVHS; to short). Here we stress only that the method à la Griffiths based on the extension of the Macaulay's theorem in algebra to the weighted case, cf. \cite{tu}, almost never works in our case.

There is a complete list of quasi-smooth Fano hypersurfaces with $\rho=1$. We have the 'famous 95', that is a list of 95 families of quasi-smooth, anti-canonically polarized Fano threefold of index 1. The list was first considered in \cite{iano}, and later on showed to be complete by \cite[Corollary 2.5]{kollar} see also cf. \cite{chen}. The 'famous 95' families  have been prominent in the context of explicit birational geometry: for a modern survey we refer to \cite{cheltsov}, while for an account of those in term of Hodge theory we refer to \cite{brown}.
%\\Looking at the list, it is easy to realize that none of the Fano threefolds in question satisfy the numerical conditions of those theorem based on a generalization of Macaulay?s theorem; see cf. theorem \ref{1} and theorem \ref{2} . Nevertheless we are able to prove
We prove:\\

\noindent
{\bf{Theorem [A]}}  {\it{Let $\mathcal{M}$ be the space of quasi-smooth hypersurface of degree d with only terminal singularities in weighted projective space $w\PP=w\PP(1, a_1, a_2, a_3, a_4)$, where $d=\sum_{i=1}^{4}a_i$, modulo automorphisms of  $w\PP$. Then there is an open dense subset of $\mathcal{M}$ on which infinitesimal Torelli holds.}}

\medskip
See Theorem \ref{index1}.

The situation changes drastically when we go to the higher index case. There are 35 families of Fano threefolds of higher index, that we can find listed for example in \cite{okada}, see our Table (\ref{table1}). Out of the 35 , we have $8$ out of them that are rigid and with $h^{2,1}=0$. These are, in particular, the families no. 105, 113, 119, 120, 123, 126, 129, 130, and we will call them \emph{Hodge-rigid}.
We are left with 27 families, and they split with respect to the infinitesimal Torelli property. We prove:\\

\noindent
{\bf{Theorem [B]}} {\it{ Let $\mathcal{M}$ be the space of quasi-smooth weighted hypersurfaces of degree $d$ in weighted projective $4$-space and of index $>1$ modulo automorphisms. Then 
\begin{enumerate}
\item \rm for families no. 115, 121, 122 and 127 \it the infinitesimal Torelli theorem does not hold;
\item the families \rm{no. 105, 113, 119, 120, 123, 126, 129, 130} \it are Hodge-rigid;
\item for the remaining \rm 23 \it  families, there is an open dense subset of $\mathcal{M}$ on which the infinitesimal Torelli holds.
\end{enumerate}}}

\medskip
See Theorem \ref{torelliantitorelli}.

The behaviour with respect to the infinitesimal Torelli problem is summarized in Table (\ref{table1})

%\begin{enumerate}
%\item the families \rm{no. 105, 113, 119, 120, 123, 126, 129, 130} \it are Hodge-rigid;
%\item for the families \rm no. 115, 121, 122 and 127 \it the local Torelli theorem does not holds;
%\item for the remaining \rm 23 \it  families,there is an open dense subset of $\mathcal{M}$ on which the infinitesima Torelli holds.
%\end{enumerate}

We point out that for the four (AT)-cases of Table (\ref{table1}), that is in the cases of failure for the Torelli problem, we are also able to construct explicitly the deformations along which infinitesimal Torelli fails; we call them {\it{anti-Torelli}} deformations. Indeed we get explicit polynomial basis for the vector subspace of 'anti-Torelli' deformations; see subsection (\ref{explicit}).

 \subsection{GM-type infinite towers}
 We consider as one of the main result of this work the explicit construction of infinite towers of quasi-smooth varieties which exhibit a failure of infinitesimal Torelli and simultaneously a behaviour analogous to the Gushel-Mukai varieties deeply studied in a list of recent papers, for example \cite{debarre2}, \cite{debarreilievmanivel}, \cite{GM1}, \cite{GM2}.\\
 Recall first that a Gushel-Mukai variety (in the following, shortened as GM) is the datum of the intersection of the cone over the Grassmannian $G$=Gr(2,5) with an appropriate linear space, and a general global section of $\of_G(2)$. 
% According to the position of the linear subspace (passing or not through the vertex) we might have two cases (for $n=3,4,5,6$) 
% 
% \begin{description}
%\item[type A]either a smooth linear section by a $\PP^{n+4}$ of the intersection of G $\subset \PP^9$ with a quadric;
%\item[type B] or a double cover of a smooth linear section by a $\PP^{n+3}$ of G$\subset \PP^9$, branched along its intersection with a quadric, 
%\end{description}
%with the type B variety forming a proper subfamily of the type A ones, except that in case $n=6$ all GM are of type B.\\
%%GM varieties have been studied in great details over the past years by several people, including Debarre, Iliev, Kuznetsov, Manivel (see for example \cite{GM1}, in particular from the point of view of rationality and hyperk\"ahler geometry. \\
One curious feature of the GM variety is that the Hodge numbers satisfy a periodic behaviour: the even dimensional varieties of GM-type will satisfy (with one further dimension to remove in the K3 case) $$H^{\textrm{dim}(X)}_{\prim}(X^{\text{even}}, \C) \cong \C \oplus \C^{20} \oplus \C^{}$$ while the odd one $$H^{\textrm{dim}(X)}(X^{\text{odd}}, \C) \cong \C^{10} \oplus \C^{10}.$$ 
%So for example, a GM variety of dimension 3 (a Fano threefold of index 1 and degree 10) will have as Hodge diamond
% \[ \begin{matrix}
%0 && 10 && 10 &&0& \\
%&0 &&1&&0&\\
%&&0&&0&&\\
%&&&1 &&&
%\end{matrix}\]
%while if $X$ is a general GM of dimension 4 then we will have
%\[ \begin{matrix}
%
%0 &&1 && 22 &&1 && 0 \\
%& 0 && 0 && 0 &&0& \\
%&&0 &&1&&0&&\\
%&&&0&&0&&&\\
%&&&&1 &&&&
%\end{matrix}\]
The fact that the Hodge structure on $H^4(X,\mathbb C)_{{\rm{}}}$ is of ${\rm{K}}3$ type and that the study of the period map for $X$ shares many similarities to the case of cubic fourfold add importance to the study of GM-varieties. 
%Moreover letting $\mathcal{X}_n$ be the family of Gushel-Mukai $n$-folds, $n=3,4$ or $5$
%it is known that the sub-family of Gushel-Mukai $n$-folds of type $B$ is a subfamily of $\mathcal{X}_n$ which is birational to $\mathcal{X}_{n-1}$. This length $3$ tower structure among Gushel-Mukai varieties makes the study of their periods very interesting.
\medskip

In section \ref{GM} we perform a geometric construction that yields examples that shares common similarities with the above varieties. Although our ambient space is not a Grassmannian variety, we decided to adopt the name of Gushel-Mukai like since we always find Hodge similarities like above, and because the geometrical construction is similar to the GM case.\\ In particular, starting from one of the hereby considered Fano threefold $X_d \subset w\PP$ (or, better, a surface lying in it as quasi-linear section)  we realize a double cover $\varphi: Y \to w\PP$ branched over $X_d$. Iterating this process we get an infinite chain (or 'tower') of varieties, satisfying the periodic equality (with the dichotomy odd/even) of the Hodge groups like in the GM case. Moreover, since (unlike in the GM case) there is no Grassmannian to 'bound' the dimension, our towers can go up and produce examples in any dimension. If the 'even' Hodge structure comes from the Hodge structure of a (weighted) K3 surface, we say that our tower is of \emph{K3-type}.\\
Let us call an \emph{even member} of the tower an $X_d^{2k}$ obtained by doing an even number of step in the construction above, similarly we will define the \emph{odd members}. In Theorem \ref{periodic} we show:
\\

\noindent
{\bf{Theorem [C]}} {\it{Let $X^{0}_d=V(f_0)\subset \PP(a_0,\ldots,a_n)$. Suppose that $d \equiv 0 \ (mod \ 2)$. Let  $X_d^{2k}$ be any even member of the tower, of dimension $n+2k$. There is an isomorphism of IVHS $$\phi: H^{n}(X_d^{0},\C) \stackrel{\sim}{\longrightarrow} H^{n+2k}(X_d^{2k},\C)[-2k].$$ In particular the central Hodge numbers of $X_d^0$ are the same of the Hodge numbers of $X_d^{2k}$ up to a degree $k$ shift, that is  $$\left( h^{n+2k,0}_{X^{2k}_d},h^{n+2k-1,1}_{X^{2k}_d},\dots,h^{1,n+2k-1}_{X^{2k}_d},h^{0,n+2k}_{X^{2k}_d}\right)=\left(0,\dots,0,h^{n,0}_{X^0_d},\dots,h^{0,n}_{X^0_d},0, \dots,0\right),$$ with $2k$ zeros on the last vector. The same holds for odd members, with an equality between the Hodge theory of $X_d^1$ and $X_d^{2k+1}$, for any $k$.}}\\ 

%Finally it always holds that $H^1(T_{X^0_d})_{\text{proj}} \cong H^1(T_{X^k_d})_{\text{proj}}$ , where $H^1(T_{X^k_d})_{\text{proj}} \cong H^1(T_{X^k_d}) $ for dimension $>2$. }}\

%In the paper the reader can find the description of all the towers arising by quasi-smooth terminal Fano hypersurfaces, see tables \ref{table2},\ref{tableind1}.

One of the most interesting case is the construction of an infinite tower whose initial odd element is a general quartic double solid \ref{quarticdouble}; there is a strong renew of interest about the (ir)rationality problem on quartic double solid, see \cite{cheltsov2}, \cite{yuri}, after Voisin's result that for any integer $k = 0,\ldots,7$, a very general nodal quartic double solid with k nodes is not stably rational; see: \cite{voisin2}. Actually we show that up to dimension $4$ onwards the members of our tower are rational; see Proposition \ref{razraz}. 

Nevertheless, for us, the most beautiful case is the following one which, in a further analogy with the GM case, holds an infinite series of examples and counterexamples to the Torelli problem according to the parity of the dimension.\\

\noindent
 {\bf{Theorem [D]}}{\it{ Any quasi-smooth odd dimensional member $X_{14}^{2k+1} \subset \PP(2,3,4,5,7^{2k+1})$ is of Anti-Torelli type, while any even dimensional quasi-smooth member  $X_{14}^{2k} \subset \PP(2,3,4,5,7^{2k})$ is of Torelli type. In particular we have an infinite chain of GM-type varieties which are examples or counterexamples for the Torelli problem, with alternate dimensions.}}

\medskip
See Theorem \ref{excounterex}.
 
We start now with a recap of known results on weighted projective varieties and Torelli problem.

\section{Preliminaries}
\subsection{Weighted projective hypersurfaces and Jacobian ring}\label{subsectioninfinitesimaltorelli} We denote with $\PP(a_0, \ldots, a_n)$ or $w\PP$ the weighted projective space with variables $x_0, \ldots, x_n$ with wt ($x_i)=a_i$. For the standard theory of weighted projective spaces we refer to \cite{miglia} and \cite{dolgachev82}.\\
We recall that the weighted projective space $\PP(a_0, \ldots, a_n)$ is said to be \emph{well-formed} if $$\textrm{hcf}(a_0,\ldots, \hat{a_i}, \ldots, a_n)=1 \ \textrm{ for each } i$$
We want now to define a reasonable class of varieties in a (well-formed) weighted projective space, that will play the role of the smooth one in straight case. We have the following definition
\begin{defi} Let $X$ be a closed subvariety of a well-formed weighted projective space, and let us denote by $A_X$ the affine cone over $X$, that is the completion of $A^{\bullet}_X = \pi^{-1}(X)$, where $\pi: \A^{n+1} \smallsetminus 0  \longrightarrow w\PP$ is the canonical projection. We say that $X$ is \emph{quasi-smooth} if $A_X$ is smooth outside its vertex $\underline{0}$.
\end{defi}
The above definition is telling us that the only (possible) singularities of $X$ come just from the automorphisms of the ambient space. The key idea is that a quasi-smooth subvariety $X \subset w\PP$ is a \emph{V-manifold}, that is isomorphic to the quotient of a complex manifold by a finite group of holomorphic automorphisms. 
%In particular, as shown by Dolgachev \cite{dolgachev82} to a quasi-smooth subvariety is canonically associated a pure Hodge Theory, defined analogously to the smooth case. 
Operative definitions of quasi-smoothness can be found, for example, in \cite[Theorem 8.1]{iano}.
%: here we quote the criterion valid for hypersurface, that we will use for all ours computations.

%\begin{thm}[\cite{iano}, Theorem 8.1]
%  Let $X$ be a general hypersurface in $\PP(a_0\ldots a_n)$, of degree
 % $d$: it is quasismooth if and only if one of the following
 % conditions is satisfied.
 % \begin{enumerate}
 % \item There exists a variable $x_i$, for some $i$, of weight $d$,
  %and in this case we say that $X$ is a \emph{linear cone}.
 % \item For every non-empty subset $I=\lbrace i_0\ldots
   %i_{k-1}\rbrace\subseteq\lbrace 0\ldots n\rbrace$, one of the following holds:
  %  \begin{enumerate}
 %   \item   there   exists   a   monomial   $x_I^M=x_{i_0}^{m_0}\cdots
 %     x_{i_{k-1}}^{m_{k-1}}$ of degree $d$;
 %   \item for $\mu=1\ldots k$ there exist monomials
   % \begin{displaymath}
    %x_I^{M_\mu}x_{e_\mu}=
     %x_{i_0}^{m_{0,\mu}}\cdots x_{i_{k-1}}^{m_{k-1,\mu}}\, x_{e_\mu},
   %\end{displaymath}
    %of degree $d$, where $\lbrace e_\mu \rbrace$ and $k$ are distinct elements.
   %\end{enumerate}
 % \end{enumerate}
%\end{thm}

%In particular Torelli problems for quasi-smooth weighted hypersurfaces represents a natural class to consider. As we are going to explain in the next sections, the problem has been solved  by Donagi and Tu (cfr. \cite{tu}) in many numerical case, namely under some numerical coincidences between the weights $a_i$ and the degree $d$. The problem remained however still open for an important classes of quasi smooth weighted hypersurfaces, that is the case of $\Q$-Fano Threefolds. 

A $\Q$-Fano 3-fold here means a
projective 3-fold with only terminal singularities whose anticanonical divisor is ample.
%A $\Q$-Fano 3-fold is an important object in the classification theory of algebraic
%3-folds. It is one of the end products of the Minimal Model Program. 
We assume as well (this being always satisfied in the hypersurface case) that the rank $\rho$ of its second Betti homology is $1$.
The Fano index $\iota_X$ of a Fano 3-fold X is

$$\iota_X=\textrm{max} \lbrace m \in \Z >0 \ | \ -K_X =mA \textrm{ for some Weil divisor A } \rbrace$$
%$$\iota_X\textrm{max} \lbrace m \in \Z  0 \ | \ − K_X  mA \textrm{ for some Weil divisor A } \rbrace $$
 A Weil divisor $A$ for which $-K_X =\iota_X A$ is called a primitive ample divisor.\\
 Recall that if $X_d=V(f)$ is a hypersurface of degree d, we can associate to $X_d$ a canonical deformation ring, the Jacobian ring
   $$R_f:=\mathbb C[x_0,\ldots,x_{n+1}]/J(f),$$
   where $J(f)$ is its Jacobian ideal, that is, in both smooth and quasi-smooth cases, the (graded homogeneous) ideal generated by the partial derivatives of $f$.\\
   This is indeed finite dimensional as vector space if and only if $X$ is quasi-smooth. The Jacobian ring be interpreted as the ring of infinitesimal deformation of the affine cone $A_X$ (see \cite{io} for a more general treatment).
%    In particular we have an isomorphism of graded ring $$ R_f(d) \cong T^1_{A_X}$$ where the latter can be identified with $\Ext^1_{\of_{A_X}}(\Omega^1_{A_X}, \of_{A_X})$. Note that this a priori is a module, but one can show that in the case of an hypersurface acquires a ring structure (thanks to local duality). Note that in particular $$(T^1_{A_X})_0 \cong (R_f)_d$$ (with the subscript we mean that homogeneous component), and the left hand side can be interpreted (see \cite{schlessinger}) as the subvector space $H^1(X,T_X)_{{\rm{proj}}}\subset H^1(X,T_X)$ given by those deformations which preserve the projective embedding $X\hookrightarrow w\PP$.\\
A well-known feature of the Jacobian ring of a hypersurface is that some homogeneous slices of it can be identified with the (primitive) cohomology of the underlying $X$. This is the content of the \emph{Griffiths Residue theorem} (\cite{dolgachev82} or for a modern account, see \cite{carlson}). In particular we have that, if $H^{p,n-p}_{\prim}(X)$ is the $(p,n-p)$-primitive part of $H^{n}(X,\mathbb C)$, then
$$H^{p, n-p}_{\prim}(X) \cong R_{(n-p+1)d-s},$$ $s=\sum_{i=0}^{n+1}a_i$ is the sum of the weights and $d$ is the degree of $f$.\\

Griffiths description of the primitive cohomology for a smooth hypersurface is a powerful tool to attack the Torelli problem. The problem was originally solved by Donagi, cf. \cite{donagi}. A modern survey can be found, for example, in Claire Voisin's book \cite{voisin2}, that here we recall briefly.\\
 Assume that $X_d$ is a smooth hypersurface of degree $d$ in $\PP^{n+1}$. Assume that the dimension of $X$ is at least 3 (in particular, all the deformations of $X$ are projective). Thanks to the Lefschetz hyperplane the only interesting part of the Hodge structure of $X$ is located in degree $n$. The starting point is realizing how the differential of the period map for an Hodge structure of weight $n$ restricted to its primitive subspaces $$\dd \wp: H^1(T_X) \rightarrow \bigoplus_{p=0}^n \Hom(H^{p,q}_{\prim}(X), H^{p-1,q+1}_{\prim}(X)),$$
 can be rewritten thanks to the Griffiths description as
 $$ R_d \longrightarrow \bigoplus_{p=0}^n \Hom(R_{(p+1)d-n-2}, R_{(p+2)d-n-2}),$$
 where $R$ denote as usual the Jacobian ring and the subscript refers to its homogeneous components.\\
We can rephrase this using local duality theorem (Theorem 2.2 of \cite{tu}).
First recall the concept of \emph{socle}. For a graded $k$-Algebra $A$, the socle is defined as $$ \textrm{Soc}(A)= \lbrace  h \in A \ | \ hg=0 \textrm{ for all } g \in \bigoplus A_i \rbrace.$$ In general the socle could be either empty or in different degrees, but in case of the Jacobian ring it coincides with a specific degree component $R_{\sigma}$, cf. Corollary A3, \cite{tu}.  One has in particular $R_{\sigma+j}=0$ for every $j>0$. The local duality theorem is then
 \begin{thm}[Thm 2.2 in \cite{tu}] \label{localduality} Let $f_0, \ldots, f_{n+1}$ a regular sequence of weighted homogeneous polynomial in $\C[x_0, \ldots, x_{n+1}]$ and let $$A=\C[x_0, \ldots, x_{n+1}]/ (f_0, \ldots, f_{n+1}).$$ Suppose $a_i$ is the weight of $x_i$ and $d_i= $ deg $f_i$. Then for any $0 \leq a \leq \sigma$ the pairing given by multiplication $$ A_a \times A_{\sigma-a} \longrightarrow A_{\sigma}$$ is non-degenerate, where $\sigma= \sum d_i-a_i$ is the top degree.
 \end{thm}
 If we pick $A=R$, the Jacobian ring, the theorem applies since the partial derivatives form a regular sequence. \\
 The above theorem allows us to check, instead of the injectivity of $$R_d \longrightarrow \Hom(R_a, R_b)$$ the surjectivity of $$R_b \times R_{\sigma-(a+b)} \to R_{\sigma-a}.$$
However, the above theorem assures the non-degeneracy of the multiplication map only when the socle is involved. \\
The non-degeneracy of the general multiplication map $$ R_a \times R_b \longrightarrow R_{a+b}$$ is indeed tackled by the \emph{Macaulay's theorem}.
\begin{thm}[Thm. 1 in \cite{tu}] Let $f_0, \ldots, f_{n+1}$ be a regular sequence of homogeneous polynomials of degree $d_0, \ldots, d_{n+1}$ in $\C[x_0, \ldots, x_{n+1}]$ and let $$R= \C[x_0, \ldots, x_{n+1}]/(f_0, \ldots, f_{n+1}).$$ Then $R$ is a finite dimensional graded $\C$-algebra with top degree $\sigma=\sum(d_i-1)$ and the multiplication map $$ \mu: R_a \times R_b \longrightarrow R_{a+b}$$ is nondegenerate for $a+b \leq \sigma$.
\end{thm}
In turn this is enough to guarantee the infinitesimal Torelli theorem for smooth projective hypersurfaces in $\PP^{n+1}$. We focus now on the case of quasi-smooth hypersurfaces in weighted projective spaces.
%The situation becomes more interesting when we consider the \emph{generic} Torelli problem. This is the question of knowing whether the period map $\wp$ is degree 1 over its image. In other words, given two hypersurfaces $Y$ and $Y'$ of degree $d$ in $\PP^{n+1}$ with $Y$ generic such that there is an isomorphism of polarised Hodge structures $$i: H^n_{\prim}(Y,\Z) \cong H^n_{\prim}(Y', \Z)$$ one asks wheter $Y$ and $Y'$ are isomorphic. Using Griffiths technique, Donagi first and Cox-Green later \footnote{Cox and Green refined Donagi's argument, solving the case $d=6, \ n \equiv 2$ mod $6$ left open by Donagi.} proved
%\begin{thm}[Thm 6.4, \cite{donagi1983generic}, Thm 1, \cite{cox1990polynomial}] The generic Torelli theorem holds for hypersurfaces of degree $d$ in $\PP^{n+1}$, up to the possible exceptions of the following cases
%\begin{itemize}
%\item $d$ divides $n+2$;
%\item $d=3, \  n+1=3$;
%\item $d=4, \ n+1 \equiv 1$ mod $ 4$;
%\end{itemize}
%\end{thm}
%
%Case (ii) represents a true exception. In case (i) the answer is expected to be positive up to a finite number of counterexamples. Cases known are cubic curves in $\PP^2$, quartic surfaces in $\PP^3$ and quintics in $\PP^5$.

\subsection{Infinitesimal Torelli for weighted hypersurfaces}
A natural extension of the Donagi work is consider the case of quasi-smooth hypersurfaces in weighted projective spaces. Let alone the generic Torelli, one fails to get a general answer even for the infinitesimal Torelli.\\The main obstacle is that the weighted version of Macaulay's theorem does not hold for any weighted projective space $\PP(a_0, \ldots, a_n)$ (more in general, for any weighted graded rings). Counterexamples can be easily given. Consider for example the case of $R=\C[x,y]/(x^2, y^3)$. The multiplication map $\mu: R_1 \times R_3 \to R_4$ is degenerate, since $x_0 \cdot R_3=0$ with $x_0 \neq 0$.\\
Nevertheless on some weighted graded rings some generalization of Macaulay's theorem holds. Based on a detailed analysis, Tu was able to identify in \cite{tu} several classes of quasi-smooth hypersurface for which the infinitesimal Torelli theorem actually holds. We recap briefly here his results.\\For $w\PP=\PP(a_0, \ldots, a_{n+1})$ set $$s= \sum a_i, \ \ m=lcm(a_0, \ldots, a_{n+1}) $$ and for any subset $J=(j_1, \ldots, j_n)$ of $\lbrace 0, \ldots, n+1 \rbrace$, $$m(q|J) =lcm(q_{j_1}, \ldots, q_{j_n}).$$ We define $$G=-s+\frac{1}{n+1}\sum_{2 \leq k\leq n+2} {n \choose k-2}^{-1} \sum_{|J|=k}m(q|J).$$ An estimate for $G$ is $$G\leq -s+m(n+1);$$ in particular we notice that for the standard projective space $\PP^{n+1}$ we have $s=n+2, m=1, G=-1$. What we have is 
\begin{thm}[Theorem 2.8 in \cite{tu}] Let $R=\C[x_0, \ldots, x_{n+1}]/J$ be the weighted ring defined by the ideal $J$ of a regular sequence $f_0, \ldots, f_{n+1}$. Set $d_i= \ddeg f_i$, $a_i=$ weight $x_i$ and $\sigma= \sum (d_i-a_i)$. The natural map $$R_a \to \Hom(R_b, R_{a+b})$$ is injective
\begin{enumerate}
\item if $b$ is a multiple of m and $\sigma-(a+b) \geq \textrm{max}(G+1,0)$, or
\item if $\sigma-(a+b)$ is a multiple of $m$ and $b \geq G+1$.
\end{enumerate}
\end{thm}

It holds:
\begin{thm}[Theorem 2.10 in \cite{tu}]\label{1} let $p$ an integer between 1 and n for which gcd(m,p) divides s. Then there are infinitely many non-negative integers $k \geq ((n+1)p/(n+1-p))-(s/m)$ for which $d=(s+km)/p$ is a positive integer. The infinitesimal Torelli theorem holds for quasi-smooth hypersurfaces of degree $d$ in $\PP(a_0,\ldots, a_{n+1}).$
\end{thm}

For other specific choices of the weights it is known:
\begin{thm}[Theorem 4.1 in \cite{donagitu}] \label{2} Let $\PP$ a weighted projective space $\PP(a_0, \ldots, a_{n+1})$ for which $a_0=a_1=1$ and $m$ divides $s$, and let $\mathcal{M}$ the space of quasi-smooth weighted hypersurfaces of degree $d$ in $\PP$ modulo automorphisms of $\PP$. Assume $d$ is a multiple of $m$ and $d \geq $max$(3s, s+m(n+1))$. Then there is an open dense subset of $\mathcal{M}$ on which the period map is defined and injective.
\end{thm}

Many interesting cases are still left open. In particular the answer was not known for $\Q$-Fano threefolds hypersurfaces, despite their importance in terms of birational geometry. In this paper we give an answer to this problem, following a careful analysis of the Jacobian rings involved.

\section{Fano hypersurfaces threefolds and Torelli problem}

We first study the infinitesimal Torelli for the 'famous 95'. These are Fano threefolds of index 1 first discovered by Iano-Fletcher and Reid. They are listed online in the graded ring database \cite{grdb}, together with several thousands of other Fanos in higher codimension. We do not include here the complete list of 95 hypersurfaces: the interested reader can easily consult the above database. 
\subsection{The infinitesimal Torelli for Fano varieties of index 1}
Recall from the introduction the Griffiths description of the differential of the period map for an Hodge structure of weight $k$ $$\dd \wp: H^1(T_X) \rightarrow \bigoplus_{p=0}^k \Hom(H^{p,q}(X), H^{p-1,q+1}(X)).$$
In particular, its injectivity is enough for the infinitesimal Torelli theorem to hold. Assume that $X_d$ is a smooth hypersurface of degree $d$ in $w\PP^n(a_0, \ldots, a_n)$ of dimension at least 3, and that $H^2(\of_X)=0$ (equivalently, all the deformations of $X$ are projective). Denote by $s= \sum a_i$. The Griffiths-Steenbrink description of the primitive cohomology of  a quasi-smooth hypersurface reduces the problem to the injectivity of the polynomial map $$ R_d \longrightarrow \bigoplus_{p=0}^n \Hom(R_{(p+1)d-s}, R_{(p+2)d-s}),$$ where $R$ denotes as usual the Jacobian ring and the subscript refers to its homogeneous components.\\
For a quasi-smooth Fano threefolds we have $H^{3,0}(X)=H^{0,3}(X)=0$: moreover if we focus on the index 1 condition what we have to verify in order to have the infinitesimal Torelli for any of the previous 95 families is to check the injectivity of the natural map $$ R_d \to \Hom (R_{d-1}, R_{2d-1}).$$
If we use local duality theorem \ref{localduality}, since the socle is located in degree $\sigma=3d-2\iota_X$ we can rephrase  the problem as follows
\begin{rmk} Let $X$ as above. The infinitesimal Torelli theorem holds for $X$ if the natural map $$ \Sym^2(R_{d-1}) \to R_{2d-2}$$ is surjective. 
\end{rmk}
%
%
%
%
%As we said, what we have to verify in order to have the infinitesimal Torelli for any of the previous 95 families is to check the injectivity of the natural map $$ R_d \to \Hom (R_{d-1}, R_{2d-1}).$$
%Equivalently, using Local duality theorem (Theorem 2.2 of \cite{tu}) what we need to verify is that the natural multiplication map
%$$ \Sym^2(R_{d-1}) \to R_{2d-2}$$ is surjective. 
We first try to apply the condition of weighted Macaulay's theorem recalled in the introduction.  Consider for example the case (5) of the list above, that is $X_{7}\subset \mathbb{P}(1,1,1,2,3)$. Using the notations of Theorem \ref{1}, the equations that need to be satisfied is $d=7=(8+6k)/p$, together with the condition $k\geq (4p/4-p)-4/3$, but this is clearly not possible. By similar standard computations we get
\begin{lemma} Only the families no.1 and no.2 of the Fletcher-Reid list (that is, $X_4 \subset \PP^4$ and $X_5 \subset \PP(1^4,2)$ ) satisfies the numerical conditions of the weighted Macaulay's theorem.
\end{lemma}

In the other cases, a partial answer can be obtained by looking at the surjectivity of the multiplication map in the ambient ring. More precisely
\begin{proposition}[Proposition 2.3 in \cite{tu}] Let $R=\C[x_0,\ldots, x_n]/J$ be a weighted ring for which local duality holds, and let $\sigma$ be the top degree of R. Given non-negative integers $a,b$ satisfying $a+b \leq \sigma$ and denoting with $S=\C[x_0,\ldots, x_n]$ if $$ S_b \times S_{\sigma-(a+b)} \to S_{\sigma-a}$$ is surjective, then $R_a\to \Hom(R_b, R_{a+b})$ is injective.
\end{proposition}

In our case, we have $a=d, b=d-1$ and $\sigma=\sum d-2a_i=5d-2s$, with $s=d+1$. Therefore we have to check the result for $S_{d-1} \times S_{d-1} \to S_{2d-2}$, or equivalently the surjectivity of the natural map $$\Sym^2(S_{d-1}) \to S_{2d-2}.$$ 
\begin{lemma}Of the 93 remaining families, only the number 5, that is $X_7 \subset \PP(1,1,1,2,3)$ satisfies the surjectivity already at the level of polynomial ring.
\end{lemma}
\begin{proof}
On $X_7$ the multiplication map  $\Sym^2(S_{6}) \to S_{12}$ can be verified to be surjective. All the other cases yields counterexamples. For example consider $X_6\subset\PP(1^4,3)$. $\Sym^2(S_{5}) \to S_{10}$ is not surjective as we can see by the element $x_4^3x_0$, where $x_4$ is the variable of weights 3. Indeed is clear that $S_5=\langle \Sym^5(x_0, \ldots, x_3), \Sym^2(x_0, \ldots, x_3) \cdot x_4 \rangle$. An extensive computer search using Macaulay2 confirms this statement.
%\begin{verbatim}
%def=(d, weights)->(
%R=QQ[x,y,z,t,w, Degrees=>weights];
%L1=flatten entries (symmetricPower(2, basis(d-1,R)));
%L2=flatten entries basis(2*d-2,R); APP=select(L2, a-> not member(a, L1));
%IL=ideal(L1); AA=ideal(select(L2, a-> not member(a, L1))); 
%print APP, print LINCOMB,for i in APP do if ((i % IL)!=0) 
%then print i else print 0
%)
%\end{verbatim}
\end{proof}

\begin{corollary} Let $X$ any quasi smooth member of the family $X_7 \subset \PP(1,1,1,2,3)$. Then the differential of the Period map $$\dd \wp: H^1(X, T_X) \to \Hom(H^{2,1}(X), H^{1,2}(X))$$ is injective, and therefore the local Torelli theorem holds.
\end{corollary}

To prove the theorem for the other cases, we have to use the following lemma
\begin{lemma} \label{genericity}Let $\pi: \mathcal{X} \to U$ a flat family of quasi-smooth hypersurfaces such that for the central fiber $X_0=\pi^{-1}(0)$ the infinitesimal Torelli property holds. Then the same holds for the fibers $\pi^{-1}(U)$ over an open $0 \in U$.
\end{lemma}
\begin{proof} We recall that $\mathcal{M}_0$ is a complex manifold and the period domain is a variety. Hence the condition of the differential of the period map being of maximal rank is an open condition.
\end{proof}

\begin{thm}\label{index1} Let $X_d \subset \PP(a_0, \ldots, a_4)$ one of the 92 families left (not 1,2,5). Then the Local Torelli theorem holds generically for $X_d$.
\end{thm}
\begin{proof}
The idea of the proof is simple. We check the surjectivity of the dual of the infinitesimal period map on a quasi-smooth element $X_0$. A genericity condition can be found for example in \cite{iano}. Once this is confirmed, by the lemma \ref{genericity} the result will hold in a Zariski open set $ 0 \in U$, where we can think of $X_0$ as central fiber of a flat family $\pi: \mathcal{X} \to U$ .\\
The check on the central fiber can be done with the help of the following Macaulay2 code
\begin{verbatim}
def=(d, weights)->(
R=QQ[x,y,z,t,w, Degrees=>weights];
 f=random(d,R); J=Jacobian matrix f;
B=(R/ideal J);L1=flatten entries (symmetricPower(2, basis(d-1,B)));
L2=flatten entries basis(2*d-2,B); 
APP=select(L2, a-> not member(a, L1));
IL=ideal(L1);  print APP, print LINCOMB,
for i in APP do if ((i % IL)!=0) then print i else print 0
)
\end{verbatim}
Actually all the computations could be solved by hand in principle. To have a concrete grasp of how this work we refer to the example \ref{example}, where we deal with the (most) interesting case of index >1. An extensive computer search with Macaulay2 confirms that this holds for any of the 92 remaining families.
\end{proof}

\begin{rmk} {\rm{It is because of Lemma \ref{genericity} that some authors adopt the somewhat confusing expression "local Torelli" instead of "infinitesimal Torelli".}}
\end{rmk}

\subsection{The infinitesimal Torelli for Fano varieties of index >1}
We now investigate the case of higher index Fano threefolds. Unlike the index 1 case, here the situation is much more various and complicated: in particular we will have (other than the Hodge-trivial examples) some families for which the infinitesimal Torelli fails, and some families for which it holds. We recall in \ref{table1} the list of Fano threefold of higher index, from Okada (see \cite{okada}), ordered according to the index.\\
\newline
\noindent\begin{minipage}{\linewidth}
   \centering
\begin{tabular}[]{cccc|cccc}
\hline
No. & $X_d \subset \mbP (a_0,\dots,a_4)$ & Ind&Torelli & No. & $X_d \subset \mbP (a_0,\dots,a_4)$  & Ind&Torelli \\
\hline
96 & $X_3 \subset \mbP (1,1,1,1,1)$ &2& T& 113 & $X_4 \subset \mbP (1,1,2,2,3)$ &  5 &R\\[0.2mm]
97 & $X_4 \subset \mbP (1,1,1,1,2)$ &  2& T& 114 & $X_6 \subset \mbP (1,1,2,3,4)$ & 5 &T\\[0.2mm]
98 & $X_6 \subset \mbP (1,1,1,2,3)$ &  2 &T & 115 & $X_6 \subset \mbP (1,2,2,3,3)$ &  5&AT \\[0.2mm]
99 & $X_{10} \subset \mbP (1,1,2,3,5)$ &2& T& 116 & $X_{10} \subset \mbP (1,2,3,4,5)$ &  5 &T\\[0.2mm]
100 & $X_{18} \subset \mbP (1,2,3,5,9)$ &  2&T & 117 & $X_{15} \subset \mbP (1,3,4,5,7)$ &  5&T \\[0.2mm] 
101 & $X_{22} \subset \mbP (1,2,3,7,11)$ &2& T& 118 & $X_6 \subset \mbP (1,1,2,3,5)$ &  6 &T\\[0.2mm]
102 & $X_{26} \subset \mbP (1,2,5,7,13)$ &  2& T& 119 & $X_6 \subset \mbP (1,2,2,3,5)$ &  7&R \\[0.2mm]
103 & $X_{38} \subset \mbP (2,3,5,11,19)$ & 2& T& 120 & $X_6 \subset \mbP (1,2,3,3,4)$ &  7 &R\\[0.2mm]
104 & $X_2 \subset \mbP (1,1,1,1,1)$ & 3 &R& 121 & $X_8 \subset \mbP (1,2,3,4,5)$ &  7 &AT\\[0.2mm]
105 & $X_3 \subset \mbP (1,1,1,1,2)$ & 3 &T& 122 & $X_{14} \subset \mbP (2,3,4,5,7)$ &  7& AT\\[0.2mm]
106 & $X_4 \subset \mbP (1,1,1,2,2)$ & 3 & T&123 & $X_6 \subset \mbP (1,2,3,3,5)$ & 8 &R\\[0.2mm]
107 & $X_6 \subset \mbP (1,1,2,2,3)$ &  3 &T &124 & $X_{10} \subset \mbP (1,2,3,5,7)$ &  8& T\\[0.2mm]
108 & $X_{12} \subset \mbP (1,2,3,4,5)$ & 3 &T& 125 & $X_{12} \subset \mbP (1,3,4,5,7)$ &  8&T \\[0.2mm]
109 & $X_{15} \subset \mbP (1,2,3,5,7)$ & 3 & T&126 & $X_6 \subset \mbP (1,2,3,4,5)$ &  9 &R\\[0.2mm]
110 & $X_{21} \subset \mbP (1,3,5,7,8)$ & 3 & T&127 & $X_{12} \subset \mbP (2,3,4,5,7)$ &9 &AT\\[0.2mm]
111 & $X_4 \subset \mbP (1,1,1,2,3)$ & 4 &T &128 & $X_{12} \subset \mbP (1,4,5,6,7)$ &  11 &T\\[0.2mm]
112 & $X_6 \subset \mbP (1,1,2,3,3)$ & 4 &T& 129 & $X_{10} \subset \mbP (2,3,4,5,7)$ &  11 &R\\[0.2mm]
& & & &130 & $X_{12} \subset \mbP (3,4,5,6,7)$ & 13& R\\
\end{tabular}
\captionof{table}{List of Fano threefolds of index 1 and their behaviour with respect to Torelli problem.  The notation AT/T/R stands for (respectively) Anti-Torelli, Torelli, Rigid}
  \label{table1} 
\end{minipage}
What do we have to check? Let us call the index of any $X$ $\iota_X$, that is $\omega_X \cong \of_X(-\iota_X)$. Keeping notations as in the introduction we have
\begin{lemma} Let $X_d \subset \PP(a_0, \ldots, a_4)$ a quasi-smooth Fano threefold hypersurface of index $\iota_X$. Then the infinitesimal Torelli theorem holds if the natural map $$R_{d-\iota_X} \times R_{d-\iota_X} \to R_{2d-2\iota_X}$$ is surjective.
\end{lemma}
\begin{proof}
The success of infinitesimal Torelli corresponds to the injectivity of the map $$R_d \to \Hom(R_{d-\iota_X}, R_{2d-\iota_X}),$$ where $R_{d-\iota_X} \cong H^{2,1}(X)$ and $R_{2d-\iota_X} \cong H^{1,2}(X) \cong (H^{2,1}(X))^{\vee}$ by Serre duality. Equivalently, using Local duality theorem (Theorem 2.2 of \cite{tu}) what we need to verify is that the surjectivity of natural multiplication map $$ R_{d-\iota_X} \times R_{\sigma-(2d-\iota_X)} \to R_{\sigma-d}.$$
On the other hand, since $s= \sum a_i=d+\iota_X$, we have $\sigma=5d-2s=3d-2\iota_X$. 
\end{proof}

We will now analyse separately the three interesting cases, that is Hodge-rigid families, the anti-Torelli and the Torelli.

\subsection{Hodge-rigid families}
Recall that we call an $X_d \subset w\PP$ \emph{Hodge-rigid} if both $H^1(X,T_X)=0$ and $H^{2,1}(X)=0$. Now, for these threefolds we do not have any Torelli-type question to ask: therefore we want to classify and remove those cases from our list. Now, we recall that from Griffiths-Steenbrink theory we have $R_d \cong H^1(X, T_X)$, and by Serre duality $$ H^1(X,\Omega^2_X)\cong H^1(T_X\otimes \omega_X)=H^1(X, T_X(-\iota_X))= R_{d-\iota_X}.$$ From this it clearly follows that if $d< \iota_X$, then $H^{2,1}(X)=0$. 
\begin{proposition} The following families satisfies $H^{2,1}(X)=H^1(X,T_X)=0$. Therefore they are Hodge-rigid.
\begin{itemize}
\item \rm no. 104 $ X_2 \subset \PP^4;$
\item no. 113 $X_4 \subset \PP(1,1,2,2,3);$
\item no. 119 $X_6 \subset \PP(1,2,2,3,5);$
\item no. 120 $X_6 \subset \PP(1,2,3,3,4);$
\item no. 123 $X_6 \subset \PP(1,2,3,3,5);$
\item no. 126 $X_6 \subset \PP(1,2,3,4,5);$
\item no. 129 $X_{10}\subset \PP(2,3,4,5,7);$
\item no. 130 $X_{12} \subset \PP(3,4,5,6,7).$
\end{itemize}
\end{proposition}
\begin{proof}
The vanishing of $H^{2,1}(X)$ is assured by the condition $d< \iota_X$ above. To verify the vanishing of $H^1(X, T_X)$, simply notice that for every member of the list above one has $d> \sigma$. Therefore, by definition of socle, $R_d=0$.
\end{proof}

We go on a bit further with the analysis of the Hodge-rigid families. First, from the classification of Okada (\cite{okada}) we see that each member of the Hodge rigid families has the general member rational, and moreover we just computed that they do not have any deformation. We want now to define a notion of \emph{strong rigidity}. Recall from \cite{badescu}, if $A_X$ denotes the affine cone over $X$, then $(T^1_{A_X})_k$, for $k<0$ parametrizes the extension of $X$ as projective variety (and as well contains the smoothings of $A_X)$. If $X=V(f)$ is a hypersurface we have $(T^1_{A_X})_k \cong (R_f)_{d-k}$ and moreover $(R_f)_{d-k} \cong H^1(X, T_X(d-k))$. We say that $X$ is \emph{strongly rigid} if $(R_f)_j=0, \ j\neq 0$. A strongly rigid hypersurface therefore will  not have any extension (rather than the trivial one). We note that $X_d$ will satisfy $(T^1_{A_X})_{-1} \cong (R_f)_{d-1} =0$. In particular we will have to add non-trivial weights in order to extend (when possible) the Hodge-rigid members. 
\begin{corollary} Amongst the Hodge-Rigid families, the \rm no.104, 126, 129 \it are strongly rigid. 
\end{corollary}\begin{proof}
The Jacobian rings of the examples above have Hilbert-Poincarè series equal to 1. Therefore they have $R_0=1$ and $R_k=0, \ k \neq 0$.
\end{proof}
\subsection{Torelli and Anti-Torelli families}
We now investigate first the families that does not satisfy the infinitesimal Torelli property.\\
     \begin{thm}\label{antitorelli}Let $X_d$ a quasi-smooth member of one of the four families \textrm{ no. 115, 121, 122, 127}. Then the infinitesimal Torelli does not hold for $X_d$.
     \end{thm}
     \begin{proof}
 We will analyse separately the four different cases.\\    
    \paragraph{\bf{{no. 115 }}} Pick any general quasi smooth member of the family of index $\iota_X=5$, $X_6 \subset \PP(1,2,2,3,3)$, where we name coordinates $x,y_0, y_1, z_0, z_1$. Since $d-\iota_X=1$ we have to check the surjectivity of the map: $R_1 \times R_1 \to R_2$. Now, since the partial derivatives form a regular sequence, if we compute the Hilbert-Poincaré series of the Jacobian ring we have $$\textrm{HP}(R)=\prod \frac{(1-t^{d-a_i})}{(1-t^{a_i})}= 1+t+3t^2+3t^3+4t^4+3t^5+3t^6+t^7+t^8.$$
Therefore, since there is only one generator of degree one, we have $\Sym^2(R_1) \cong <x^2>$, whereas $R_2$ contains, for example, $y_0, y_1$, and the natural map cannot be surjective. Dually speaking we have the standard map from $R_d$ given by $$\C^3 \cong R_6 \to \Hom(R_1, R_7)  \cong \C,$$ and this cannot be injective.\\

\paragraph{{\bf{No. 121}}} The same phenomenon occurs for $X_8 \subset \PP(1,2,3,4,5)$, of index 7, with coordinates $x,y,z,v,w$. We have indeed that the Hilbert-Poincaré series is $$HP(R)=1+t+2t^2+2t^3+3t^4+3t^5+3t^6+2t^7+2t^8+t^9+t^{10}$$ and  for $\C^2 \cong R_8 \to \Hom(R_1,R_9) \cong \C$ injectivity clearly fails.\\

\paragraph{{\bf{No. 122}}} This is $X_{14} \subset \PP(2,3,4,5,7)$, of index $\iota_X=7$. Here we have as Hilbert-Poincaré series $$HP(R)=1+t^2+t^3+2t^4+2t^5+3t^6+3t^7+5t^8+4t^9+6t^{10}+$$$$+5t^{11}+7t^{12}+6t^{13}
      +7t^{14}+6t^{15}+7t^{16}+5t^{17}+$$ $$+6t^{18}+4t^{19}+5t^{20}+3t^{21}+3t^{22}+2t^{23}+2t
      ^{24}+t^{25}+t^{26}+t^{28}.$$
 Here we have $R_{14} \cong \C^7$ and $R_7 \cong \C^3$, so we cannot conclude immediately injectivity. Nevertheless, $\Sym^2 (R_7)$ has dimension 6, so the surjectivity is excluded.\\ 
    
 \paragraph{{\bf{No. 127}}} This is $X_{12} \subset \PP(2,3,4,5,7)$. Here the Hilbert-Poincaré series is $$HP(R)=1+t^2+t^3+2t^4+t^5+3t^6+2t^7+3t^8+2t^9+3t^{10}+2t^{11}+3t^{12}+t^{13}+$$$$+2t^
     {14}+t^{15}+t^{16}+t^{18}.$$ Thus $R_{12} \cong \C^3$, while $\Hom(R_3, R_{15}) \cong \C$.\\
 
  \end{proof}
\subsubsection{The Torelli families of index $>1$}The other 24 families behave in the opposite way.
We were able to check the surjectivity statement already at the level of the ring $S$, that is $\Sym^2S_{d-\iota_X}\to S_{2d-2\iota_X}$, for the families no. 98, 111, 118, 128. This is enough to guarantee the surjectivity at the level of the Jacobian ring since $S_{2d-2\iota_X}\to R_{2d-2\iota_X}$ is the quotient map. No. 111 and No. 118 exhibits a curious behaviour. They are respectively $X_4 \subset \PP(1,1,1,2,3)$ and $X_6 \subset \PP(1,1,2,3,5)$. They both verify $d=\iota_X$ and therefore $d=\sigma$. In particular, one has to check either the surjectivity of the natural map $R_0 \times R_0 \to R_0$ or the injectivity of $R_{\sigma} \to \Hom(R_0, R_{\sigma})$, and this trivially holds. For all the other families we proceed by the means of an extensive computer search, as in the index 1 case. We recall here an example.\\
\begin{ex}\label{example} \rm{
Consider $X_6\subset \PP(1,1,2,3,4)$ in the class no.114 given by the equation $x_0^6+x_1^6+x_2^3+x_3^2+x_2x_4$. Note that $X_6$ is  quasi-smooth. The Hilbert Series of its Jacobian ring is $$
HP(R)=1+2t+3t^2+4t^3+5t^4+4t^5+3t^6+2t^7+t^8
$$
We have to check the surjectivity $\Sym^2R_1\to R_2$. The space $R_1$ is generated by $x_0,x_1$, while $R_2$ is generated by $x_0^2, x_0x_1, x_1^2$. Indeed the variable $x_2$ of weight 2 is in the Jacobian ideal. $\Sym^2R_1$ is 3-dimensional, and equal to $R_2$. In particular the multiplication map is surjective. Therefore infinitesimal Torelli holds.}
\end{ex}
The same genericity argument of Lemma \ref{genericity} yields
\begin{thm}\label{torelliantitorelli} Let $\mathcal{M}$ the space of quasi-smooth weighted hypersurfaces of degree $d$ in $\PP$ modulo automorphisms of $\PP$, for any of the non Hodge-trivial families of quasi-smooth Fano Threefolds of index $i_X >1$. Then
\begin{itemize}
\item for the families no. 115, 121, 122 and 127 the infinitesimal Torelli theorem does not hold;
\item for the remaining 24 families, there is an open dense subset of $\mathcal{M}$ on which the infinitesimal Torelli theorem holds.
\end{itemize}
\end{thm}

\subsection{Explicit Anti-Torelli deformations} \label{explicit}

Here we list again the four families of anti-Torelli Fano threefold hypersurfaces
\begin{center}
\begin{tabular}{|c|c|c|c|c|}
\hline 
No & $X_d \subset \mbP (a_0,\dots,a_4)$ & Ind. & $ h^1(T_X)$& $h^{2,1}$ \\ 
\hline 
115 & $X_6 \subset \mbP (1,2,2,3,3)$ & 5 & 3&1 \\ 
\hline 
121 & $X_8 \subset \mbP (1,2,3,4,5)$ & 7 & 2& 1 \\ 
\hline 
122 &  $X_{14} \subset \mbP (2,3,4,5,7)$  & 7 & 7& 3\\ 
\hline 
127 & $X_{12} \subset \mbP (2,3,4,5,7)$ & 9 & 3&1 \\ 
\hline 
\end{tabular} 
\end{center}
We now analyse carefully these four examples in terms of Torelli and Anti-Torelli deformations.

\paragraph{{\bf{no.115}}}
Let us start analysing the first of the example. To get explicit examples of Anti-Torelli deformations we use the following member of the family defined by the polynomial
$$f=x_0^6+x_1^3+x_2^3+x_3^2+x_4$$
where $V(f) \subset \PP(1,2,2,3,3)$. $V(f)$ is clearly quasi-smooth since $f$ is of Fermat-type. Recall by the previous section that $$\dd \wp: R_d \to \Hom(R_{d-\iota_X}, R_{2d-\iota_X})$$ cannot be injective since $R_6 \cong \C^3$ and $R_1\cong R_7 \cong \C$. Denote thus by $K_f$ the kernel of the period map associated with this specific central member of the family giving the period map. Any element in $K_f$ will give rise to a deformation of \emph{Anti-Torelli} type. In the following table we list the generators of the interesting graded component of the Jacobian ring $R_f$ and the anti-Torelli deformations associated with $f$.\\
\noindent\begin{minipage}{\linewidth}
   \centering
\begin{tabular}{|c|c|}
\hline 
$R_1$ & $\langle x_0 \rangle$ \\ 
\hline 
$R_7$ & $\langle x_0^3x_1x_2 \rangle$ \\ 
\hline 
$R_6$ & $\langle x_0^4x_1, x_0^4x_2,x_0^2x_1x_2 \rangle$ \\ 
\hline 
$K_f$ & $\langle x_0^4x_1, x_0^4x_2 \rangle $ \\ 
\hline 
\end{tabular} 
\captionof{table}{Torelli and Anti-Torelli for family no.115}
\end{minipage}
To compute $K_f$ note that both $(x_0^4x_1)x_0$ and $(x_0^4x_2)x_0$ belong to $J_f$ while $(x_0^2x_1x_2)x_0=x_0^3x_1x_2$ is the generator we already found for $R_7$. Therefore a family of Anti-Torelli deformation with central fiber the fixed $X_0=V(f)$ is given by $V(f+\lambda x_0^4x_1 + \mu  x_0^4x_2$).\\

\paragraph{{\bf{No.121}}}

Here we use as the central fiber the following member of the family defined by the polynomial
$$f=x_0^8+x_1^4+x_2x_4+x_3^2$$
where $V(f) \subset \PP(1,2,3,4,5)$. This is not Fermat-type but the Jacobian ideal is generated by $(x_0^7,x_1^3,x_4,x_3,x_2)$. Moreover the tangent monomials at the singular points induced by the ambient space are linear. Therefore the above variety is quasi-smooth. Even here $$\dd \wp: R_d \to \Hom(R_{d-\iota_X}, R_{2d-\iota_X})$$ cannot be injective since $R_8 \cong \C^2$ and $R_1\cong R_9 \cong \C$. In the following table we list the generators of the interesting graded component of the Jacobian ring $R_f$ and the anti-Torelli deformations associated with $f$.\\
\noindent\begin{minipage}{\linewidth}
   \centering
\begin{tabular}{|c|c|}
\hline 
$R_1$ & $\langle x_0 \rangle$ \\ 
\hline 
$R_9$ & $\langle x_0^5x_1^2 \rangle$ \\ 
\hline 
$R_8$ & $\langle x_0^6x_1, x_0^4x_1^2 \rangle$ \\ 
\hline 
$K_f$ & $\langle x_0^6x_1 \rangle $ \\ 
\hline 
\end{tabular} 
\captionof{table}{Torelli and Anti-Torelli for family no.121}
\end{minipage}
Here again $(x_0^6x_1)x_0 \in J_f$, while $x_0^5x_1^2$ generates $R_9$.
Therefore a family of Anti-Torelli deformation with central fiber the fixed $X_f$ will be given by $V(f+\lambda x_0^6x_1$).\\
\paragraph{{\bf{No.127}}}\label{127}
Here we use the following member of the family defined by the polynomial
$$f=x_0^6+x_1^4+x_2^3+x_3x_4$$
where $V(f) \subset \PP(2,3,4,5,7)$. Here same Jacobian criteria of the previous example apply. Even in this case $\dd \wp: R_d \to \Hom(R_{d-\iota_X}, R_{2d-\iota_X})$ cannot be injective since $R_{12} \cong \C^3$ and $R_3\cong R_{15} \cong \C$. In the following table we list the generators of the interesting graded component of the Jacobian ring $R_f$ and the anti-Torelli deformations associated with $f$.\\
\noindent\begin{minipage}{\linewidth}
   \centering
\begin{tabular}{|c|c|}
\hline 
$R_3$ & $\langle x_1 \rangle$ \\ 
\hline 
$R_{15}$ & $\langle x_0^4x_1^2 \rangle$ \\ 
\hline 
$R_{12}$ & $\langle x_0^3x_1^2, x_0x_1^2x_2, x_0^4x_2 \rangle$ \\ 
\hline 
$K_f$ & $\langle  x_0^3x_1^2, x_0x_1^2x_2 \rangle $ \\ 
\hline 
\end{tabular} 
\captionof{table}{Torelli and Anti-Torelli for family no.127}
\end{minipage}
The Kernel verifications follow as in the previous cases. Therefore a family of Anti-Torelli deformation with central fiber the fixed $X_f$ will be given by $V(f+\lambda x_0^3x_1^2+ \mu x_0x_1^2x_2  $).\\

\paragraph{\bf{No.122}}
Here we use the following member of the family defined by the polynomial
$$f=x_0^7+x_0x_2^3+x_1^3x_3+x_2x_3^2+x_4^2$$
where $V(f) \subset \PP(2,3,4,5,7)$. Looking at the equation the Jacobian ideal has no zeros. 

Unlike the previous examples, here the non-injectivity cannot be deduced a priori, but follows from a careful examination of the multiplication map. In the following table we list the generators of the interesting graded component of the Jacobian ring $R_f$ and the anti-Torelli deformations associated with $f$.\\
\noindent\begin{minipage}{\linewidth}
   \centering
\begin{tabular}{|c|c|}
\hline 
$R_7$ & $\langle x_0^2x_1, x_0x_3, x_1x_2 \rangle$ \\ 
\hline 
$R_{21}$ & $\langle x_0^4x_1x_3^2, x_0^3x_3^3, x_2^4x_3\rangle$ \\ 
\hline 
$R_{14}$ & $\langle x_0x_1x_2x_3, x_0^4x_1^2,x_0^5x_2, x_0^3x_1x_3, x_0^2x_1x_2, x_0^2x_3^2, x_1^2x_2^2 \rangle$ \\ 
\hline 
$K_f$ & $\langle  x_0x_1x_2x_3 \rangle $ \\ 
\hline 
\end{tabular} 
\captionof{table}{Torelli and Anti-Torelli for family no.127}
\end{minipage}
To check the kernel $K_f$ we proceed as follows. First we check using the previous tools that the element of $R_{2d-2\iota}$ in the cokernel of multiplication map is $x_0^5x_2$. Then we dualize via $(R_k)^{\vee} \cong R_{\sigma-k}$. Here $R_d=R_{14}=R_{\sigma-d}$. The (non-canonical) isomorphism between $R_{14}$ and itself yielding duality pairs up $x_0^5x_2$ with $\frac{1}{32}x_0x_1x_2x_3$ with respect to the socle basis $x_0^5x_1x_3^3$. One can in fact verifies that the above multiplication gives exactly the socle generator, whereas any other multiplication of $x_0^5x_2$ with any other basis element of $R_{14}$ lies in $J_f$. Indeed one can identify $\frac{1}{32}x_0x_1x_2x_3$  as dual element of $x_0^5x_2$. Therefore a family of Anti-Torelli deformation with central fiber the fixed $X_f$ will be given by $V(f+\lambda x_0x_1x_2x_3  $).\\
We therefore raise the following question:
\begin{question} For any of the above examples, consider any deformation of $X_d$ induced by one of the $\xi \in K_f$, that is of Anti-Torelli-type. Is there any \emph{geometrical} reason for the failure of the Torelli theorem in these specific directions?
\end{question}

\section{Gushel-Mukai type infinite towers}\label{GM}
We start our analysis of an interesting geometrical phenomenon which links the geometry of some weighted Fano hypersurfaces to the Gushel-Mukai geometry recalled in the introduction. We first will explain the necessary algebraic background in order to understand this strange phenomenon, and later on we will analyse in great details some example of this construction.

\subsection{Hyperplane section principle for graded rings}
The key phenomenon here is the \emph{hyperplane section principle} for graded rings. This is a standard result in commutative algebra. We will follow closely the treatment of \cite{miles}

Let $R$ be a graded ring, and $x_0 \in R$ a graded element of degree $a_0$. Suppose
that $x_0$ is a regular element of R, that is, a non-zerodivisor. Then multiplication
by $x_0$ is an inclusion $R \to R$ with image the principal ideal $(x_0)$, and
one arrives at the exact sequence$$
0 \to (x_0) \to R \to \overline{R} \to 0,$$
where $\overline{R} = R/(x_0)$. Geometrically, Proj $(\overline{R})$ is the hyperplane section of
Proj $(R)$ given by $x_0 = 0$.
The hyperplane section principle says that under these assumptions, we
can deduce a lot of the structure of $\overline{R}$ from $R$ and vice-versa.
\begin{thm} [Hyperplane section principle, \cite{miles}]
\begin{enumerate}
\item \label{1hyp}  Let $\overline{x_1}, \ldots, \overline{x_k}$ be
homogeneous element that generate $\overline{R}$, and $x_1,\ldots , x_k \in R $ any homogeneous
elements that map to $\overline{x_1}, \ldots, \overline{x_k} \in \overline{R}$. Then $R$ is generated by
$x_0, x_1,\ldots, x_k.$
\item  Under the assumption of (\ref{1hyp}), let $\overline{f_1}
,\ldots, \overline{f_n}$
be homogeneous generators
of the ideal of relations holding between $\overline{x_1}, \ldots, \overline{x_k}$. Then there exist
homogeneous relations $f_1,\ldots, f_n$ holding between $x_0, x_1, \ldots, x_n$ in $R$
such that the $f_i$ reduces to $\overline{f_i}$ modulo $x_0$ and $f_1, \ldots , f_n$ generate the
relation between $x_0, x_1, \ldots, x_n$.
\item  Similar for the syzygies.
\end{enumerate}
\end{thm}

\subsection{Construction of Towers}
Let us start from a weighted hypersurface  $X^{0}_d=V(f_0)\subset \PP(a_0,\ldots,a_n)$. Suppose that $d \equiv 0 \ (mod \ 2)$ and call $d=2t$. Consider now $$X^1_d \stackrel{2:1}{\longrightarrow} \PP(a_0, \ldots, a_n)$$ a double cover of the ambient weighted projective space, branched over $X^0_d$.
Since $d$ is even, $X^1_d$ has a model as $$X^1_d =V(f_0+y_1^2) \subset \PP(a_0, \ldots, a_n, t).$$ Let us call $f_1:=f_0+y_1^2$. Suppose now that $\omega_{X_d^0} \cong \of_{X^0_d}(m_0)$: by adjunction one has $$ \omega_{X^1_d} \cong \of_{X^1_d}(-\sum a_i-t+d) \cong \of_{X^1_d}(m_0-t)=: \of_{X^1_d}(m_1).$$ 
We have the following immediate result
\begin{lemma} $X^0_d$ is quasi-smooth if and only if $X^1_d$ is.
\end{lemma}
\begin{proof}
Simply notice that $J_{f_1}=(J_{f_0}, y_1)$. Therefore the only possible new singularities of $X^1_d$ comes just from the new ambient space.
\end{proof}
Set $R_{f_0}=S/(J_{f_0})$ and similarly $R_{f_1}$ for $X^1_d$. It is clear that $R_{f_1} \cong S[y_1]/(J_{f_0}, y_1)$, therefore by the hyperplane section principle one has $R_{f_0} \cong R_{f_1}$.
We pick any general member $X^1_d=V(g) \subset \PP(a_0, \ldots, a_n, t)$. Since the partial derivatives form a regular sequence, one clearly has $\ddim (R_{f_1})_k= \ddim (R_{g})_k$, for all $k$ and we have an isomorphism of $\C$-vector spaces between the two Jacobian rings. Now we see that by completing the square the "double cover type" does not form a proper subfamily inside the space of all $X^1_d  \subset \PP(a_0, \ldots, a_n, t).$
% Indeed given any $g=y_1^2+h_1 \cdot y + f_0$ we can complete the square, and realizing $X^0_d$ as appropriate hyperplane section.\\
Clearly the process can go on to give an infinite chain of double covers
\begin{equation}
\xymatrix{
&&&\cdots\ar[d]^-{\varphi_3}\\
&&X_d^2\ar[d]^-{\varphi_2}\ar@{^{(}->}[r]& w\PP(\underline{a},t,t)\\
&X_d^1\ar[d]^-{\varphi_1}\ar@{^{(}->}[r]& w\PP(\underline{a},t)&\\
X_d^0\ar@{^{(}->}[r] & w\PP(\underline{a})\ &&}
\end{equation}

%\[
%\begindc{\commdiag}[200]
%\obj(5,7)[k]{$...$}
%\obj(5,5)[a1]{$w\PP(\underline{a},t,t)$}
%\obj(3,5)[a]{$ X_d^2\subset$}
%\obj(3,3)[c1]{$w\PP(\underline{a},t)$}
%\obj(1,3)[c]{ $ X_d^1\subset$}
%\obj(1,1)[e]{$X_d^0\subset w\PP(\underline{a})$}
%\mor{k}{a1}{$\varphi_3$}
%\mor{a}{c1}{$\varphi_2$}
%\mor{c}{e}{$\varphi_1$}
%\enddc
%\]
where any $X^j_d$ is a double cover of the projective space $\PP(a_0, \ldots, a_n, t^{j-1})$ branched on $X^{j-1}_d$.
By hyperplane section principle
\begin{proposition} \label{jacobian} For any $X^j_d$ obtained with tower construction one has $R_{f_j} \cong R_{f_0}$.
\end{proposition}
We point out that we can reverse the chain, and considering $X_d^1$ as an extension of $X_d^0$. Recall that, in general, $Y\subset \PP^{n+1}$ is said to be an \emph{extension} of $X \subset \PP^n$ if we have dim($X)=n$, dim$(Y)=n+1$ and there exists an immersion $i: \PP^n \to \PP^{n+1}$ such that $X=Y \cap i(\PP^n)$. 
%In the (straight) projective space, extensions of projective varieties are contained in the degree -1 parte of the $T^1_{A_X}$ (see for example \cite{abs} in the case of curves, or \cite{badescu} for an account of the classical theory). 
In the weighted case of $X \subset w\PP(a_0, \ldots, a_n)$ we thus have to look of the several graded components $(T^1_{A_X})_{-a_i}$. Moreover, for the case of a hypersurface of degree $d$ we have $$T^1_{A_X} \cong R_f(d).$$ To sum up: given $X_d^0$ we also have $X_d^1$ as $$X_d^1=V(f_0+h_t\cdot y_1) \subset \PP(a_0, \ldots, a_n,t),$$ where $y_1$ is a new variable of weight $t$ and more important $$h_t \in (T^1_{A_{X_d^0}})_{-t}\cong(R_{f_0})_{d-t} \cong(R_{f_0})_{t}.$$ It is trivial to see that there is no difference between this and a double cover model.\\
Instead what is not clear a priori is what happens to the Hodge groups when we run up the tower.
\subsection{Periodic Patterns in Hodge Theory}
Let us call an \emph{even member} of the tower an $X_d^{2k}$ obtained by doing an even number of step in the construction above. Similarly we will define the \emph{odd members}. 
\begin{thm} \label{periodic}\begin{itemize}
\item Let $X_d^{2k}$ any even member of the tower, of dimension $n+2k$. We have that $$H^{n+2k}(X_d^{2k}, \C) \cong H^{n}(X_d^0, \C).$$ The isomorphism is compatible with the Hodge decomposition: in particular the central Hodge numbers of $X_d^0$ are the same of the Hodge numbers of $X_d^{2k}$ up to a degree $k$ shift, that is  $$\left( h^{n+2k,0}_{X^{2k}_d},h^{n+2k-1,1}_{X^{2k}_d},\dots,h^{1,n+2k-1}_{X^{2k}_d},h^{0,n+2k}_{X^{2k}_d}\right)=\left(0,\dots,0,h^{n,0}_{X^0_d},\dots,h^{0,n}_{X^0_d},0, \dots,0\right),$$ with $2k$ zeros on the last vector;
\item the same holds for odd members, with an equality between the Hodge numbers of $X_d^1$ and $X_d^{2k+1}$, for any $k$.
\end{itemize}
\end{thm}
\begin{proof}This is just a careful analysis of the involved components of the Jacobian ring. Let us start from the base of our chain, $X^0_d \subset \PP(a_0, \ldots, a_n).$ We will denote as before $d=2t$ and $s=\sum a_i$ Since by Proposition \ref{jacobian} the Jacobian ring will be the same in any step of our construction, we will drop the subscript referring to the equation and denote it simply with $R$. By Griffiths-Steenbrink one has $$H^{n-p,p}_{\prim}(X_d^0)\cong R_{(p+1)d-s}$$
Now extend to $X_1^d \subset \PP(a_0, \ldots, a_n,t)$. By the same argument we will have
$$H^{(n+1)-p,p}_{\prim}(X_d^1)\cong R_{(p+1)d-s-t}.$$
If we extend another time we have
$$H^{(n+2)-p,p}_{\prim}(X_d^2)\cong R_{(p+1)d-s-2t} \cong R_{pd-s} \cong H^{n-p+1,p-1}_{\prim}(X_d^0),$$
where we set $H^{r,s}_{\prim}(X)=0$ if $r<0$ or $s<0$. On the other end the same yields for $X_d^1$ with
$$H^{(n+3)-p,p}_{\prim}(X_d^3)\cong R_{(p+1)d-s-3t} \cong R_{pd-s-t} \cong H^{(n+1)-p+1,p-1}_{\prim}(X_d^1).$$
Of course every time we perform a double extension we end up in the same graded components of the Jacobian ring, just with shift in the Hodge theory, exactly as concluded above.  \end{proof}
\begin{rmk} Note in particular that, since the degree of any member of the tower is constant, we will always have $H^1(T_{X^k_d})_{\text{proj}} \cong R_d$. We recall that the distinction between $H^1(T_{X^k_d})_{\text{proj}}$ and $H^1(T_{X^k_d})$ holds only in dimension $\leq 2$.
\end{rmk}
\begin{rmk} As showed for example in \cite{saito}, the Jacobian ring of a variety actually determines its (IVHS). Therefore what we have is indeed an isomorphism of IVHS $$\phi: H^{n}(X_d^{0}) \stackrel{\sim}{\longrightarrow} H^{n+2k}(X_d^{2k})[-2k].$$
\end{rmk}

\subsection{Gushel-Mukai-type geometry}
Let us now focus on our case, that is, Fano threefolds of index $>1$, and pick any $X_d \subset \PP(a_0, \ldots, a_n)$. To be in a tower one of the weights $a_i$ has to be $a_i=d/2$. If not, one can start running the game directly from $X_d=X_d^0$, adjoining a variable of half the weight of the degree. Of course if $d \equiv 1 \ (\mod 2)$ there is no hope of building any tower. \\ Looking at the table of Okada, it turns out that 30 families of Fano threefolds of index $>1$ lies in a tower. At the end of the section we include a table with all the towers, and the relevant Hodge groups in both even and odd case. Note that to be Fano of K3 type, they have to satisfy $t=\iota_X$, where $t$ is the covering variable as above. Among all, 4 of them are of K3-type, and we will focus our attention on these in particular. 
They corresponds to the families \begin{enumerate}
\item No. 97, $X_4 \subset \PP(1,1,1,1,2)$;
\item No. 107, $X_6 \subset \PP(1,1,2,2,3)$;
\item No. 116, $X_{10} \subset \PP(1,2,3,4,5)$;
\item No. 122, $X_{14} \subset \PP(2,3,4,5,7)$.
\end{enumerate}

%{\bf{Periodic pattern of cohomologies}}
%Take $X\subset \PP(\underline{a},t)$ of degree $d=2t$ and odd dimension $n$. Define $s$ as the sum of the weights, then $H^{p,n-p}(X)\cong R_{(n-p+1)d-s}$. 
%Extend $X$ twice to obtain a quasi-smooth $Z\subset \PP(\underline{a},t^3)$ of dimension $n+2$. Then we have $H^{p+2,n-p}(Z)\cong R_{(n-p)d-s}\cong H^{p+1,n-p-1}(X)$ since the Jacobian ring of $X$ is the same as the Jacobian ring of $Z$. Hence the Hodge numbers of $Z$ are the same of $X$ with an appropriate shift, that is $(h^{n+2,0}_Z,h^{n+1,1}_Z,\dots,h^{1,n+1}_Z,h^{0,n+2}_Z)=(0,h^{n,0}_X,\dots,h^{0,n}_X,0)$.
%In the same way, take $Y\subset \PP(\underline{a},t^2)$ of degree $d$ and even dimension and $W\subset \PP(\underline{a},t^4)$ its extension. Their Hodge numbers are related in the exact same way as in the odd case. Hence for our tower it is enough to find the Hodge numbers of the threefold $X$ and of its first extension $Y$. 

\subsubsection{A non K3-type example}Before analysing the K3-type examples, let us do a non K3-one. Pick the family no.115, already considered before for being a counterexample to the infinitesimal Torelli problem.  Again, this is the family of index $\iota_X=5$, $X_6 \subset \PP(1,2,2,3,3)$. The starting point of this tower is therefore the curve $X_6^0 \subset \PP(1,2,2)$, with $\omega_{X_6^0} \cong \of_{X_6^0}(1)$. The Hilbert-Poincaré series of the Jacobian ring of  the tower is again
$$\textrm{HP}(R)=\prod \frac{(1-t^{d-a_i})}{(1-t^{a_i})}= 1+t+3t^2+3t^3+4t^4+3t^5+3t^6+t^7+t^8.$$
In particular any odd member will have an Hodge structure (coming from the one of a curve) concentrated in degree 1 and 7 (note that dim($R_1)$=dim($R_7$)=1), while any even member will have the Hodge structure in degree -2, 4, 6, with $R_2=R_6=0$ and dim($R_4$)=4. So, for example, the Fano threefold will have as Hodge diamond
 \[ \begin{matrix}
0 && 1 && 1 &&0& \\
&0 &&1&&0&\\
&&0&&0&&\\
&&&1 &&&
\end{matrix}\]
and the Fano fourfold will have
\[ \begin{matrix}

0 &&0 && 5 &&0 && 0 \\
& 0 && 0 && 0 &&0& \\
&&0 &&1&&0&&\\
&&&0&&0&&&\\
&&&&1 &&&&
\end{matrix}\]

We start now discussing the geometry of the GM-like varieties of the K3 type. 

\subsubsection{Tower on quartic double solid: a GM-like Fano of K3 type} \label{quarticdouble}
The first example of GM-like Fano that we are going to consider is $X_4\subset \PP(1^{4},2)$ with coordinates $x_0, \ldots, x_3,y_1$, already famous in literature as the quartic double solid. It is a double cover of $\PP^3$ ramified on a quartic (K3) surface. Note that the threefold itself is smooth (and not only quasi-smooth): in fact, the generic member of the family will have $y_1^2+ \ldots$, and therefore will avoid the coordinate point $P_4= [0,0,0,0,1]$.\\
An infinitesimal Torelli theorem for the quartic double solid was already established by Clemens (\cite{clemens}). Moreover, since it shares numerical coincidences (in particular, the dimension of the intermediate Jacobian) with the Gushel-Mukai Fano threefold of index 10, $Y_{10}$, Tyurin conjectured the existence of a birational isomorphism between $Y_{10}$ and $X_{4}$. This was just recently disproved by Debarre, Iliev, Manivel in \cite{debarreilievmanivel}.\\
If we consider the double cover $X_4^2$ of $\PP(1,1,1,1,2)$ branched on the quartic double solid (we use this notation because we consider the K3 surface as the base of the tower) we will have that the resulting variety will be quasi-smooth, acquiring  $2 \times \frac{1}{2}(1,1,1,1)$ points on the intersection with the weighted $\PP(2,2)$. If we compute the Hilbert-Poincaré series of the Jacobian ring of any member of the tower, this will be  $$
\textrm{HP}(R)=
1+4t+10t^2+16t^3+19t^4+16t^5+10t^6+4t^7+t^8.$$
The odd Hodge structure will be concentrated in degrees 2 and 6 (with $R_2\cong R_6$ both 10 dimensional), while the even Hodge structure will be in degrees $0, 4, 0$. It follows that for example the Fano threefold will have as Hodge Diamond
 \[ \begin{matrix}
0 && 10 && 10 &&0& \\
&0 &&1&&0&\\
&&0&&0&&\\
&&&1 &&&
\end{matrix}\]
and the Fano fourfold will have
\[ \begin{matrix}

0 &&1 && 20 &&1 && 0 \\
& 0 && 0 && 0 &&0& \\
&&0 &&1&&0&&\\
&&&0&&0&&&\\
&&&&1 &&&&
\end{matrix}\]
\noindent
and the same for every even (and odd) dimension. Note that even in higher dimension the period map will depend only by the K3 structure (see for example \cite{debarre2}): in particular any member of the tower will satisfy the Torelli property. A further result regards the rationality property: while it is known that a smooth quartic double solid is irrational (\cite{voisin}), the same does not hold in higher dimensions. We prove in fact
 
\begin{proposition}\label{razraz} Let $X^j_4 \subset \PP(1^4, 2^j)$ a quasi-smooth member of the tower of quartic double space with dim $X^j_4 \geq 4$. Then $X^j_4$ is rational. 
\end{proposition}
\begin{proof} Indeed any quasi-smooth $X^j_4 \subset \PP(1^4, 2^j)$ is given up to  $\PP(1^4, 2^j)$-automorphism by an equation like:
$$
X^j_4: =( f(x_0,x_1,x_2,x_3,y_1,\ldots,y_{j-2})+y_{j-1}y_j=0)
$$
\noindent
where $f(x_0,x_1,x_2,x_3,y_1,\ldots,y_{j-2})$ is a quasi-smooth polynomial of degree $4$. Now let us consider the open sub-scheme $U_{j}\hookrightarrow \PP(1^4, 2^j)$ given by $y_j=1$. 
We set $V_j:= U_j\cap X^j_4$. By definition there exists a birational morphism between the affine variety $\mathbb A^{j-2}_{x_0,x_1,x_2,x_3,y_1,\ldots,y_{j-2}}\dashrightarrow V_j$ given by 
$$\phi\colon (x_0,x_1,x_2,x_3,y_1,\ldots,y_{j-2})\mapsto (x_0,x_1,x_2,x_3,y_1,\ldots,y_{j-2}, f(x_0,x_1,x_2,x_3,y_1,\ldots,y_{j-2})).
$$
\end{proof}

\subsubsection{A second example of GM-like Fano of K3 type}
The other example of GM-like Fano of K3 type that we want to investigate is the family no.122, $X_{14} \subset \PP(2,3,4,5,7)$. This is particularly interesting, since is the only GM-like of K3 type that represents as well a counterexample to the infinitesimal Torelli problem. We have already investigated the Hilbert-Poincaré series of the Jacobian ring, this being $$HP(R)=1+t^2+t^3+2t^4+2t^5+3t^6+3t^7+5t^8+4t^9+6t^{10}+$$$$+5t^{11}+7t^{12}+6t^{13}
      +7t^{14}+6t^{15}+7t^{16}+5t^{17}+$$ $$+6t^{18}+4t^{19}+5t^{20}+3t^{21}+3t^{22}+2t^{23}+2t
      ^{24}+t^{25}+t^{26}+t^{28}.$$
 Here we see that the Hodge theory in even dimension is concentrated in degree $1,14$, while in odd dimension in degree $7, 21$.
 It follows that for example the Fano threefold will have as Hodge Diamond
 \[ \begin{matrix}
0 && 3 && 3 &&0& \\
&0 &&1&&0&\\
&&0&&0&&\\
&&&1 &&&
\end{matrix}\]
and the Fano fourfold will have
\[ \begin{matrix}

0 &&1 && 8 &&1 && 0 \\
& 0 && 0 && 0 &&0& \\
&&0 &&1&&0&&\\
&&&0&&0&&&\\
&&&&1 &&&&
\end{matrix}\]
with $8$ being exactly the Picard rank of the K3 surface that is the step 0 of the tower. A curious behaviour appears here: one has in fact

\begin{thm}\label{excounterex}Any odd dimensional member $X_{14}^{2k+1} \subset \PP(2,3,4,5,7^{2k+1})$ will be of Anti-Torelli type, while any even dimensional member  $X_{14}^{2k} \subset \PP(2,3,4,5,7^{2k})$ will be of Torelli type. In particular we have an infinite chain of examples and counterexamples for the Torelli problem, with alternate dimensions.
\end{thm}
\begin{proof}
The failure in any odd dimension follows from the same reasons explained in \ref{antitorelli}. On the other hand in even dimension one has to check the (trivial) injectivity of the map $$ \dd \wp: R_{14} \longrightarrow \Hom(R_0, R_{14} \oplus\C) \cong \Hom(\C, R_{14} \oplus\C).$$
The result follows immediately then.
\end{proof}
This type of construction can be obviously performed for any quasi-smooth hypersurface we considered before. However, we decided to focus in full details only in these three examples, since they shared the most interesting geometric properties. In particular we plan to continue the study of  these GM-like Fano varieties of K3-type in a forthcoming work. As an example of our computations, we include now the complete list of towers coming out from the Fano threefolds of index >1.
\newpage
\noindent\begin{minipage}{\linewidth}
   \centering 
\begin{tabular}{|c|c|c|c|}
\hline 
No & $X_d \subset \mbP (a_0,\dots,a_4)$ & Odd &  Even  \\ 
\hline
97& $X_4\subset \mbP (1,1,1,1,2)$ & 0,10,10,0& 0,1,20,1,0\\ 
%\hline
98& $X_6\subset \mbP (1,1,1,2,3)$ & 0,21,21,0& 0,3,37,3,0\\ 
%\hline
99& $X_{10}\subset \mbP (1,1,2,3,5)$ & 0,38,38,0& 0,7,63,7,0\\ 
%\hline
100& $X_{18}\subset \mbP (1,2,3,5,9)$ & 0,49,49,0& 0,10,79,10,0\\ 
%\hline
101& $X_{22}\subset \mbP (1,2,3,7,11)$ & 0,65,65,0& 0,14,103,14,0\\
%\hline
102& $X_{26}\subset \mbP (1,2,5,7,13)$ & 0,66,66,0& 0,14,105,14,0\\
%\hline
103& $X_{38}\subset \mbP (2,3,5,11,19)$ & 0,45,45,0& 0,10,71,10,0\\   
%\hline
104& $X_{2}\subset \mbP (1,1,1,1,1)$ & 0,0,0,0& 0,0,2,0,0\\ 
%\hline
106& $X_{4}\subset \mbP (1,1,1,2,2)$ & 0,3,3,0& 0,0,8,0,0\\ 
%\hline
107& $X_{6}\subset \mbP (1,1,2,2,3)$ & 0,8,8,0& 0,1,17,1,0\\ 
%\hline
108& $X_{12}\subset \mbP (1,2,3,4,5)$ & 0,19,19,0& 0,3,33,3,0\\ 
%\hline
111& $X_{4}\subset \mbP (1,1,1,2,3)$ & 0,1,1,0& 0,0,4,0,0\\ 
%\hline
112& $X_{6}\subset \mbP (1,1,2,3,3)$ & 0,4,4,0& 0,0,9,0,0\\ 
%\hline
113& $X_{4}\subset \mbP (1,1,2,2,3)$ & 0,0,0,0& 0,0,2,0,0\\ 
%\hline
114& $X_{6}\subset \mbP (1,1,2,3,4)$ & 0,2,2,0& 0,0,6,0,0\\ 
%\hline
115& $X_{6}\subset \mbP (1,2,2,3,3)$ & 0,1,1,0& 0,0,5,0,0\\ 
%\hline
116& $X_{10}\subset \mbP (1,2,3,4,5)$ & 0,6,6,0& 0,1,13,1,0\\ 
%\hline
118& $X_{6}\subset \mbP (1,1,2,3,5)$ & 0,1,1,0& 0,0,3,0,0\\ 
%\hline
119& $X_{6}\subset \mbP (1,2,2,3,5)$ & 0,0,0,0& 0,0,4,0,0\\ 
%\hline
120& $X_{6}\subset \mbP (1,2,3,3,4)$ & 0,0,0,0& 0,0,2,0,0\\ 
%\hline
121& $X_{8}\subset \mbP (1,2,3,4,5)$ & 0,1,1,0& 0,0,4,0,0\\ 
%\hline
122& $X_{14}\subset \mbP (2,3,4,5,7)$ & 0,3,3,0& 0,1,8,1,0\\ 
%\hline
123& $X_{6}\subset \mbP (1,2,3,3,5)$ & 0,0,0,0& 0,0,1,0,0\\ 
%\hline
124& $X_{10}\subset \mbP (1,2,3,5,7)$ & 0,2,2,0& 0,0,5,0,0\\ 
%\hline
125& $X_{12}\subset \mbP (1,3,4,5,7)$ & 0,3,3,0& 0,0,7,0,0\\ 
%\hline
126& $X_{6}\subset \mbP (1,2,3,4,5)$ & 0,0,0,0& 0,0,2,0,0\\ 
%\hline
127& $X_{12}\subset \mbP (2,3,4,5,7)$ & 0,1,1,0& 0,0,3,0,0\\ 
%\hline
128& $X_{12}\subset \mbP (1,4,5,6,7)$ & 0,1,1,0& 0,0,3,0,0\\ 
%\hline
129& $X_{10}\subset \mbP (2,3,4,5,7)$ & 0,0,0,0& 0,0,3,0,0\\ 
%\hline
130& $X_{12}\subset \mbP (3,4,5,6,7)$ & 0,0,0,0& 0,0,1,0,0\\ 
\hline
\end{tabular}
\captionof{table}{Towers from Fano threefolds of index >1}
  \label{table2} 
\end{minipage}

\section{Acknowledgements}
The authors are members of INdAM--GNSAGA.

This research is supported by MIUR funds, PRIN project \lq\lq Geometry of Algebraic Varieties 2016--2018\rq\rq\ (2015EYPTSB - PE1), by DIMA project Geometry PRID ZUCCONI 2017 and by University of Udine, DIMA, XXIX-Ph.D.-programme.

The authors would like to thank Miles Reid for his useful comments.




\end{document}